\documentclass[12pt]{amsart}
\usepackage{amssymb}
\usepackage{amsmath, amscd}
\usepackage{amsthm}
\usepackage[table,xcdraw]{xcolor}
\usepackage{ulem}
\usepackage{tikz}
\usepackage{circuitikz}
\usetikzlibrary{positioning}
\usepackage[colorlinks=true, linkcolor=blue,urlcolor=blue]{hyperref}

\newtheorem{theorem}{Theorem}[section]

\newtheorem{proposition}[theorem]{Proposition}
\newtheorem{lemma}[theorem]{Lemma}

\newtheorem{corollary}[theorem]{Corollary}
\theoremstyle{definition}

\newtheorem{example}[theorem]{Example}

\newtheorem{problem}[theorem]{Problem}

\date\today

\begin{document}

\author[P. Danchev]{Peter Danchev}
\address{Institute of Mathematics and Informatics, Bulgarian Academy of Sciences, 1113 Sofia, Bulgaria}
\email{danchev@math.bas.bg; pvdanchev@yahoo.com}

\author[M. Esfandiar]{Mehrdad Esfandiar}
\address{Department of Mathematics and Computer Science, Shahed University Tehran, Iran}
\email{mehrdad.esfandiar@shahed.ac.ir}	

\author[O. Hasanzadeh]{Omid Hasanzadeh}
\address{Department of Mathematics, Tarbiat Modares University, 14115-111 Tehran Jalal AleAhmad Nasr, Iran}
\email{o.hasanzade@modares.ac.ir; hasanzadeomiid@gmail.com}

\title[Rings whose non-units are uniquely strongly clean]{\small Rings whose non-invertible elements are \\ uniquely strongly clean}

\maketitle

\begin{abstract}
We define and investigate in details the class of so-termed {\it GUSC} rings, that are those rings whose non-invertible elements are uniquely strongly clean. These rings are a common non-trivial generalization of the so-called {\it USC} rings, introduced by Chen-Wang-Zhou in J. Pure \& Appl. Algebra (2009), which are rings whose elements are uniquely strongly clean. These rings also properly generalize the so-named {\it GUC} rings, defined by Guo-Jiang in Bull. Transilvania Univ. Bra\c{s}ov (2023), which are rings whose non-invertible elements are uniquely clean.
\end{abstract}

\section{Introduction and Basic Concepts}

Everywhere in the present paper, let $R$ be an associative but {\it not} necessarily commutative ring having identity element, usually stated as $1$. Standardly, for such a ring $R$, the letters $U(R)$, $Nil(R)$ and $Id(R)$ are designed for the set of invertible elements (also termed as the unit group of $R$), the set of nilpotent elements and the set of idempotent elements in $R$, respectively. Likewise, $J(R)$ denotes the Jacobson radical of $R$, and $Z(R)$ denotes the center of $R$. We also write $ucn(R)$ for the set of all uniquely clean elements of $R$ as well as the ring of $n\times n$ matrices over $R$ and the ring of $n\times n$ upper triangular matrices over $R$ are denoted by $M_{n}(R)$ and $T_{n}(R)$, respectively.

Traditionally, a ring is said to be {\it abelian} if each of its idempotents is central, that is, $Id(R)\subset Z(R)$. For all other unexplained explicitly notions and notations, we refer to the classical source \cite{61} or to the cited in the bibliography research sources. For instance, consulting with \cite[Definition 3.4]{5}, we set $$ucn_{0}(R):=\{ e+j:e^{2}=e\in Z(R), j\in J(R)\}.$$

In order to share our achievements here, we now need the necessary background material as follows: Mimicking \cite{8}, an element $a$ from a ring $R$ is called {\it clean} if there exists $e\in Id(R)$ such that $a-e\in U(R)$. Then, $R$ is said to be {\it clean} if each element of $R$ is clean. In addition, $a$ is called {\it strongly clean} provided $ae=ea$ and, if each element of $R$ are strongly clean, then $R$ is said to {\it strongly clean} too.
On the other hand, imitating \cite{3}, $a\in R$ is called {\it uniquely clean} if there exists a unique $e \in Id(R)$ such that $a-e \in U(R)$. In particular, a ring $R$ is said to be {\it uniquely clean} (or, shortly, just {\it UC}) if every element in $R$ is uniquely clean.

Generalizing these concepts, in \cite{2} was defined an element $a \in R$ to be {\it uniquely strongly clean} if there exists a unique $e \in Id(R)$ such that $a-e \in U(R)$ and $ae=ea$. In particular, a ring $R$ is {\it uniquely strongly clean} (or, shortly, just {\it USC}) if each element in $R$ is uniquely strongly clean.

Continuing in a similar vein this terminology, expanding the first part of the presented above notions, in \cite{1} a ring is called a {\it generalized uniquely clean} ring (or just {\it GUC} for short) if any non-invertible element is uniquely clean.

\medskip

Our work in this article targets to extend considerably the last definition by defining and exploring the following key instrument: A ring is called a {\it generalized uniquely strongly clean} ring (or just {\it GUSC} for short) if every non-invertible element is uniquely strongly clean.

The relationships between all of these ring classes are visualized via the next diagram:

\bigskip

\begin{center}
	\begin{tikzpicture}
		\node[draw, minimum width=2cm, minimum height=1cm, text width=1.75cm, align=center]  (a) {USC};
		\node[draw, minimum width=2cm, minimum height=1cm, text width=2.5cm, align=center, right=of a](b){GUSC};
		\node[draw, minimum width=2cm, minimum height=1cm, text width=1.75cm, align=center, below=of b](c){GUC};
		\node[draw, minimum width=2cm, minimum height=1cm, text width=1.75cm, align=center, right=of b](d){Strongly clean};
		\node[draw, minimum width=2cm, minimum height=1cm, text width=1.75cm, align=center, below=of a](f){UC};
		\draw[-stealth] (a) -- (b);
		\draw[-stealth] (c) -- (b);
		\draw[-stealth] (f) -- (a);
		\draw[-stealth] (f) -- (c);
		\draw[-stealth] (b) -- (d);
	\end{tikzpicture}
\end{center}

\bigskip

Our motivating tool here is to develop a substantial study on this {\bf new} class of rings showing that their complete classification is extremely difficult and rather complicated. Some connections with other clean-like concepts can be found in \cite{K1,K2}. For alternative perspectives of transversal of the clean-type objects with graph theory, carrying a structural significance that extends well beyond their purely ring-theoretic role, we refer the interested readers to \cite{B} and \cite{H}, respectively.

\section{Examples and Main Results}

We begin our considerations by marking the following four simple but useful observations.

\begin{example}\label{exam1.7}
\begin{enumerate}
\item
Any UC ring is USC, but the converse is {\it not} true in general. In fact, the ring $T_{2}(\mathbb{Z}_{2})$ is USC that is {\it not} UC.

\item
Any UC ring is GUC, but the converse is {\it not} true in general. Indeed, the ring $\mathbb{Z}_{5}$ is GUC that is {\it not} UC.

\item
Any USC ring is GUSC, but the converse is {\it not} true in general. In fact, the ring $\mathbb{Z}_{(3)}=\left\{ \dfrac{a}{b} ~ \Big| ~ a,b \in \mathbb{Z}, (3,b)=1\right\}$ is GUSC that is {\it not} USC.

\item
Any GUC ring is GUSC, but the converse is {\it not} true in general. Indeed, the ring $M_{2}(\mathbb{Z}_{2})$ is GUSC that is {\it not} GUC.

\item
Any GUSC ring is strongly clean, but the converse is {\it not} true in general. To demonstrate that, for an arbitrary prime number $p$, let $\hat{\mathbb{Z}_{p}}$ be the ring of $p$-adic integers. In view of \cite[Theorem 2.4]{10}, the ring $M_{2}(\hat{\mathbb{Z}}_{p})$ is strongly clean which is manifestly {\it not} USC in virtue of \cite[Corollary 7]{2}. Also, a plain check directly shows that it is {\it not} GUSC.
\end{enumerate}
\end{example}

We, furthermore, continue with a series of helpful technicalities.

\begin{lemma}\label{lem2.1}
An element $a$ in a ring $R$ is USC if, and only if, so is $1-a$ in $R$.
\end{lemma}

\begin{proof} Straightforward.
\end{proof}

\begin{lemma}\label{lem2.2}
If $R$ is a GUSC ring, then $R$ is strongly clean.
\end{lemma}

\begin{proof}
Choosing $a \in R$, so $a \in U(R)$ or $a \not\in U(R)$. If $a \in U(R)$, then $a$ is known to be a strongly clean element. If, however, $a\not\in U(R)$, then by assumption $a$ is USC element, and thus $a$ is a strongly clean element. Finally, in both cases, $R$ is strongly clean, as expected.
\end{proof}

\begin{lemma}\label{lemma 2.6}
Let $R$ be a GUSC ring with $2 \in U(R)$ and, for every $u \in U(R)$, we have $u^2 = 1$. Then, $R$ is a commutative ring.
\end{lemma}

\begin{proof}
Firstly, we show that \( R \) is reduced, i.e., \( R \) contains no non-trivial nilpotent elements, so suppose \( q \in Nil(R) \). Then, \( (1 \pm q) \in U(R) \), whence
\[ 1 - 2q + q^2 = (1-q)^2 = 1 = (1+q)^2 = 1 + 2q + q^2. \]
Thus, \( 4q = 0 \). Since \( 2 \in U(R) \), we conclude \( q = 0 \). Hence, \( R \) is reduced and, consequently, \( R \) is abelian.
	
Moreover, for any \( u, v \in U(R) \), we have \( u^2 = v^2 = (uv)^2 = 1 \). Therefore, \( uv = (uv)^{-1} = v^{-1} u^{-1} = vu \), and hence the units commute with each other.
	
In addition, let \( x, y \in R \). We consider the following four possible cases:
	
1. \( x, y \in U(R) \): since the units commute, it must be that \( xy = yx \).
	
2. \( x, y \notin U(R) \): since \( R \) is a GUSC ring, there exist \( e, f \in Id(R) \) and $u, v\in U(R)$ such that \( x = e+u \) and \( y = f+v \). So, \( xy = yx \), because \( R \) is abelian and the units commute.
	
3. \( x \in U(R) \) and \( y \notin U(R) \): in this case, there exist \( e \in Id(R) \) and $u\in U(R)$ such that \( y = e+u \). Then, $xy=x(e+u)=xe+xu$. But, we have $xy=ex+ux=(e+u)x=yx$, because $R$ is abelian and the units commute.
	
4. \( x \notin U(R) \) and \( y \in U(R) \): similarly to case (3), we can easily see that \( xy = yx \).	

This completes our arguments after all.
\end{proof}

\begin{lemma}\label{lem2.3}
If $\displaystyle\prod_{i \in I} R_i$ is GUSC, then each direct component $R_i$ is GUSC too.
\end{lemma}

\begin{proof}
Choose an arbitrary $a_{i}\in R_{i}$, where $a_{i}\not\in U(R_{i})$, whence $$(1,1,\ldots , a_{i}, 1, \ldots )\not\in U(\prod R_{i}).$$ Thus, $(1,1,\ldots ,a_{i},1,\ldots )$ is USC and hence $a_{i}$ is USC, as required.

If, however, some element $a_{i}$ is not USC, then we have two different strongly clean decompositions for $a_{i}$ and so we have two different strongly clean decompositions for $(1,1,\ldots ,a_{i},1,\ldots )$. This shows the desired contradiction, as needed.
\end{proof}

We note that the opposite of Lemma \ref{lem2.3} is manifestly false. To exhibit that, both $\mathbb{Z}_{2}$ and $\mathbb{Z}_{3}$ are GUSC. But, the direct product $\mathbb{Z}_2 \times \mathbb{Z}_3$ is {\it not} GUSC, because the element $(0,2)$ is obviously {\it not} invertible in $\mathbb{Z}_2 \times \mathbb{Z}_3$ and $(0,2)=(1,0)+(1,2)=(1,1)+(1,1)$ are two different strongly clean decompositions of $(0,2)$ in $\mathbb{Z}_2 \times \mathbb{Z}_3$, so that $(0,2)$ is really {\it not} USC element, as wanted.

\begin{proposition}\label{prop2.5}
The direct product $\displaystyle\prod_{i\in I} R_i$ (where $|I| \geq 2$) is GUSC if, and only if, each direct component $R_i$ is USC.
\end{proposition}

\begin{proof}
Letting each $R_{i}$ be USC, then $\prod R_{i}$ is USC employing \cite[Example 2]{2} and hence $\prod R_{i}$ is trivially GUSC.

Conversely, assume that $\prod R_{i}$ is GUSC and that $R_{j}$ is {\it not} USC for some index $j\in I$. Then, there exists $a\in R_{j}$ which is {\it not} USC. Hence, $(0,\ldots ,0,a,0,\ldots )$ is {\it not} USC in $\prod R_{i}$. However, it is readily seen that $(0,\ldots ,0,a,0,\ldots )$ is {\it not} invertible in $\prod R_{i}$ and thus, by hypothesis, $(0,\ldots ,0,a,0,\ldots )$ is {\it not} USC -- an absurd. Therefore, every $R_{i}$ is USC, as wanted.
\end{proof}

Three consequences validate now.

\begin{corollary}
Let $L=\prod_{i \in I} R_i$ with $|I| \geq 2$ be the direct product of rings $R_i \cong R$. Then, $L$ is a GUSC ring if, and only if, $L$ is a USC ring if, and only if, $R$ is USC.
\end{corollary}

\begin{corollary}
For any $n \geq 2$, the ring $R^n$ is GUSC if, and only if, $R^n$ is USC if, and only if, $R$ is USC.
\end{corollary}

\begin{corollary}\label{cor2.6}
Let $R$ be a ring and $0\neq e \in Id(R)\cap Z(R)$. If $R$ is GUSC, then so is $eRe$. In addition, if $e$ is non-trivial, i.e., $e\not= 0,1$, then $eRe$ is USC.
\end{corollary}

\begin{proof}
Assume $R$ is GUSC. We have decomposed $$R=eRe \oplus (1-e)R(1-e).$$ Hence, $eRe$ is GUSC applying Lemma \ref{lem2.3}. If $e$ is {\it not} trivial, then $eRe$ is USC utilizing Proposition \ref{prop2.5}, as claimed.
\end{proof}

We next proceed by proving the following slight refinement of the last corollary.

\begin{proposition}\label{1}
Let $R$ be a ring and $0\neq e \in Id(R)$. If $R$ is GUSC, then so is $eRe$.
\end{proposition}

\begin{proof}
Suppose $R$ is a GUSC ring, so it is strongly clean in view of Lemma~\ref{lem2.2} and hence $eRe$ is strongly clean viewing \cite[Theorem 1.1.12]{13}.

Now, choose $a \in eRe$ to be not invertible in $eRe$. Write $$a=e_{1}+u_{1}=e_{2}+u_{2},$$ where $e_{1}$, $e_{2}$ are idempotents in $eRe$ and $u_{1}$, $u_{2}$ are units in $eRe$. Thus, $v_{i}=u_{i}+(1-e)$ are units in $R$ for $i=1,2$. So, $e_{1}+v_{1}=e_{2}+v_{2}$ are two strongly clean decompositions in $R$. However, as $a$ is chosen to be non-invertible in $eRe$, it must be that $e_{i}+v_{i}$ is not invertible in $R$ for $i=1,2$. Then, by hypothesis, we deduce $e_{1}=e_{2}$. Hence, $eRe$ is GUSC, as asserted.  	
\end{proof}	

Before continuing further, a useful example to know is the following one.

\begin{example}\label{exam2.7}
If $R$ is a local ring, then $R$ is GUSC.
\end{example}

\begin{proof}
Let $a \in R\setminus U(R)$. So, one inspects that $a=1+(a-1)$ is the only strongly clean decomposition for the element $a$, and thus $a$ is GUSC, as requested.
\end{proof}

However, the converse of the Example~\ref{exam2.7} is false. For instance, it is clear that $T_{2}(\mathbb{Z}_{2})$ is GUSC, but is obviously {\it not} local.

\medskip

We remember that a ring $R$ is called {\it semi-potent} if every one-sided ideal {\it not} contained in $J(R)$ contains a non-zero idempotent. In addition, a semi-potent ring $R$ is called {\it potent} if all idempotents of $R$ lift modulo $J(R)$. Semi-potent rings and potent rings were also named as {\it $I_{0}$-rings} and {\it $I$-rings}, respectively, by Nicholson in \cite{11}. Recall, moreover, that in \cite{55} a ring $R$ is called {\it exchange}, provided that, for any $a$ in $R$, there is an idempotent $e\in R$ such that $e\in aR$ and $e\in (1-a)R$ (compare with \cite{3} as well).

\begin{proposition}\label{prop2.8}
Let $R$ be a ring and $Id(R)=\{0,1\}$. Then, the following items are equivalent:
\begin{enumerate}
\item
$R$ is a local ring.
\item
$R$ is a GUSC ring.
\item
$R$ is a strongly clean ring.
\item
$R$ is a clean ring.
\item
$R$ is an exchange ring.
\item
$R$ is a potent ring.
\item
$R$ is a semi-potent.
\item
$R$ is a GUC ring.
\end{enumerate}
\end{proposition}

\begin{proof}
(i) $\Rightarrow$ (ii). It is quite elementary owing to Example \ref{exam2.7}.\\
(ii) $\Rightarrow$ (iii). It is pretty easy according to Lemma \ref{lem2.2}.\\
(iii) $\Rightarrow$ (vi) $\Rightarrow$ (v) $\Rightarrow$ (vi) $\Rightarrow$ (vii) $\Rightarrow$ (i). All evident.\\
(i) $\Rightarrow$ (viii). It is apparent that an analogous proof to that of Example \ref{exam2.7} works.\\
(viii) $\Rightarrow$ (i). Let $a\in R$ with $a\not\in U(R)$. It is now sufficient to prove that $a-1\in U(R)$. To this goal, by hypothesis, $a$ is uniquely clean, so we have $a=e+u$, where $e\in {Id}(R)=\{0,1\}$ and $u\in U(R)$. If $e=0$, it must be that $a=u\in U(R)$, a contradiction. If $e=1$, we infer $a=1+u$ and hence $a-1=u\in U(R)$, as desired.
\end{proof}

The next necessary and sufficient condition somewhat completely characterizes GUC rings in terms of GUSC rings.

\begin{proposition}\label{prop2.9}
A ring $R$ is GUC if, and only if, $R$ is GUSC and abelian.
\end{proposition}

\begin{proof}
Suppose foremost that $R$ is GUSC and abelian. So, one observes that $R$ is either local or USC. If $R$ is local, then $R$ is UC in accordance with \cite[Example 2.8]{1} whence it is GUC.

If, however, $R$ is USC, then $R$ is UC thanks to \cite[Example 4]{2} and hence $R$ is GUC. Consequently, in both cases, $R$ is GUC.

The other implication is straightforward.
\end{proof}

Our presentation in what follows unambiguously illustrates that the complete characterization of non-abelian GUSC rings is insurmountable at this stage, so we perceive in the sequel only partial descriptions of these rings.

\medskip

The following crucial observation seems {\it not} to appear in \cite{1}.

\begin{proposition}\label{pro2.2}
In any ring, idempotents and nilpotent elements are USC elements.
\end{proposition}

\begin{proof}
Let $e$ be an arbitrary idempotent element in a ring $R$. Then, we have that $e=(1-e)+(2e-1)$ is a strongly clean decomposition for $e$. Now, suppose that $e=f+u$ another strongly clean decomposition for $e$. It suffices to show only that $f=1-e$. In fact, we know that $(e-f)^2\in Id(R)\cap U(R)$ and hence $(e-f)^2=1$. So, $f=(2f-1)e+1$. Therefore, $$f=(2f-1)f=(2f-1)((2f-1)e+1)=(2f-1)(2f-1)e+(2f-1)=e+2f-1,$$ whence, $f=1-e$, as sufficed.

Suppose now that $q\in Nil(R)$ is chosen arbitrarily. Then, we have that $q=1+(q-1)$ is a strongly clean decomposition for $q$. Letting $q=e+u$ be an another strongly clean decomposition, it needs to show only that $e=1$. In fact, assume that $q^n=0$ for some natural $n>1$, and so $(e+u)^n=0$. Since $eu=ue$, one checks that $u^n\in eR$. That is why, $e=1$, as needed.
\end{proof}

Our pivotal demonstration is this one (compare with Example~\ref{exam1.7}(iv) too).

\begin{example}\label{exmnew}
The ring $M_2(\mathbb{Z}_2)$ is a GUSC ring, however is {\it not} a USC ring. This enables us that, under some extra conditions, matrix rings can be GUSC rings, but one knows that matrix rings are {\it never} USC rings.

Indeed, what we claim follows from the equality $$M_2(\mathbb{Z}_2) = U(M_2(\mathbb{Z}_2)) \cup Id(M_2(\mathbb{Z}_2)) \cup Nil(M_2(\mathbb{Z}_2))$$ combined with Proposition~\ref{pro2.2} that $M_2(\mathbb{Z}_2)$ is GUSC, as asked.

Actually, it readily follows that $M_2(\mathbb{Z}_p)$ is GUSC $\iff$ $p=2$, as well as that $M_n(\mathbb{Z}_p)$ is {\it not} USC whenever $n\geq 2$ and $p\geq 3$. To see the latter, just consider the matrix having main diagonal $(2, 0, 0 , \ldots, 0)$ and zeros elsewhere -- it obviously possesses two strongly clean representations, namely the one as the sum of the identity matrix (i.e., the matrix with main diagonal $(1, 1, 1, \ldots , 1)$ and 0s elsewhere) plus a matrix containing main diagonal as $(1, -1, -1, \ldots , -1)$ and 0s elsewhere, and the other as the sum of a matrix with main diagonal $(0, 1, 1, \ldots , 1)$ and 0s elsewhere plus a matrix with main diagonal $(2, -1, -1, \ldots , -1)$. Thus, with these two strongly clean decompositions at hand, the claim sustained.
\end{example}

With these arguments at hand, one can extract the following interesting affirmation.

\begin{corollary}\label{cor2.11}
For each $k \in \mathbb{N}$, the ring $M_2(\mathbb{Z}_{2^k})$ is GUSC.
\end{corollary}

We now come to our first major statement which, in view of what we have established above, sounds somewhat a bit curiously in regard to the logical conjecture that $M_n(\mathbb{Z}_2)$ could be GUSC for any $n\geq 2$.

\begin{theorem}\label{3}
For any ring $S \neq 0$ and any integer $n\ge 3$, the ring $M_n(S)$ need not be GUSC.
\end{theorem}

\begin{proof}
Since it is well known that $M_3(S)$ is isomorphic to a corner subring of $M_n(S)$, we just need to establish that $M_3(S)$ is {\it not} a GUSC ring invoking Proposition~\ref{1}. To this purpose, consider the matrix
$$A =\begin{pmatrix}
		1 & 1 & 0 \\
		1 & 0 & 0 \\
		0 & 0 & 0
\end{pmatrix} \notin U(M_3(S)).$$ A quick inspection shows that \( A \) admits two strongly clean decompositions, namely
\[
A =
\begin{pmatrix} 0 & 0 & 0 \\ 0 & 0 & 0 \\ 0 & 0 & 1 \end{pmatrix}
+
\begin{pmatrix} 1 & 1 & 0 \\ 1 & 0 & 0 \\ 0 & 0 & -1 \end{pmatrix}
=
\begin{pmatrix} 1 & 0 & 0 \\ 0 & 1 & 0 \\ 0 & 0 & 1 \end{pmatrix}
+
\begin{pmatrix} 0 & 1 & 0 \\ 1 & -1 & 0 \\ 0 & 0 & -1 \end{pmatrix},
\]
where, in each case, the first term is an idempotent and the second term is a unit.

This automatically gives us that $A$ is {\it not} a uniquely (strongly) clean element, as aimed. Therefore, $R$ cannot be a GUSC ring, as promised.
\end{proof}

The next construction is also worthy of recording.     	

\begin{example}\label{exam2.13}
Let $R$ be a ring. If $T_{n}(R)$ is a GUSC ring, then $R$ is also GUSC. The converse holds, provided $R$ is commutative and USC (or just is commutative and UC).
\end{example}

\begin{proof}
Let $T_{n}(R)$ be GUSC and put $e:={\rm diag}(1,0,\ldots ,0)\in T_{n}(R)$. Then, one sees that $R\cong eT_{n}(R)e$, so that $R$ is a GUSC ring using Proposition \ref{1}.

The reciprocal claim is clear exploiting \cite[Theorem 10]{2}.
\end{proof}

We continue our work with the following assertions, the first of which is immediate but worthy of mentioning.

\begin{proposition}\label{prop2.14}
For any ring $R$, the polynomial ring $R[x]$ is not GUSC.
\end{proposition}

\begin{proof}
We know that $R[x]$ is not clean, so it is {\it not} strongly clean and thus is definitely {\it not} GUSC.
\end{proof}

The next claim surprisingly does {\it not} appear in \cite{1}.

\begin{lemma}\label{lem2.18}
If $R$ is a USC ring, then $ucn(R)=ucn_0(R)$.
\end{lemma}

\begin{proof}
For any ring $R$, we know that $$ucn_0(R)=\left\{e+j \mid e^2=e \in Z(R), j \in J(R)\right\}.$$ We are also aware that $R$ is a clean ring, and thus idempotents lift modulo $J(R)$. Moreover, the quotient $\dfrac{R}{J(R)}$ is boolean bearing in mind \cite[Corollary 18]{2}, so that the equality $$ucn\left(\dfrac{R}{J(R)}\right)=ucn_0\left(\dfrac{R}{J(R)}\right)$$ forces $ucn(R)=ucn_{0}(R)$ knowing \cite[Proposition 3.6]{1}, as stated.
\end{proof}

We are now ready to address and resolve in the affirmative \cite[Question 4.8]{1} asking: Is it true that $R$ is a GUC ring if, for every non-invertible element $a\in R$, there exists an unique idempotent $e\in R$ such that $a-e\in J(R)$?

\begin{proposition}\label{prop2.21}
The answer to this question is positive.
\end{proposition}

\begin{proof}
First, we menage to show that $Nil(R)\subseteq J(R)$. In fact, let $a\in Nil(R)$.

If, for a moment, $a\in U(R)$, then $a=0$ and hence $a\in J(R)$.

If, however, $a\not\in U(R)$, there will exist an unique idempotent $e\in R$ such that $a-e\in J(R)$. Then, we have $a+J(R)=e+J(R)$, and thus $$a+J(R)\in Id(\dfrac{R}{J(R)}) ~ \bigvee ~ a+J(R)\in Nil(\dfrac{R}{J(R)}).$$ So, it must be that $a+J(R)=J(R)$ whence $a\in J(R)$.

Now, we intend to prove that, for every non-invertible element $a\in R$, there exists an unique central idempotent $e\in R$ such that $a-e\in J(R)$. In fact, let $a\in R$, where $a\not\in U(R)$, so by hypothesis there is an unique idempotent $e\in R$ such that $a-e\in J(R)$. It now suffices to establish that $e\in Z(R)$. To this aim, for each $r\in R$, $$er(1-e), (1-e)re\in Nil(R)\subseteq J(R)$$ and $$e+(1-e)re, e+er(1-e)\in Id(R),$$ so we deduce $$a-(e+er(1-e)),a-(e+(1-e)re)\in J(R).$$ However, by the uniqueness, we derive $$e+(1-e)re=e+er(1-e)=e.$$ Therefore, $er=ere=re$ and hence $e\in Z(R)$, as we want. Finally, working with \cite[Theorem 3.5]{1}, we can conclude that $R$ is GUC, as we promised above.
\end{proof}

Let $R$ be a ring and $M$ a bi-module over $R$. The trivial extension of $R$ and $M$ is defined as
\[ T(R, M) = \{(r, m) : r \in R \vee m \in M\}, \]
with addition defined component-wise and multiplication defined by
\[ (r, m)(s, n) = (rs, rn + ms). \]
The trivial extension $T(R, M)$ is isomorphic to the subring consisting of matrices of the sort
\[ \left\{ \begin{pmatrix} r & m \\ 0 & r \end{pmatrix} : r \in R \text{ and } m \in M \right\} \]
of the formal $2 \times 2$ matrix ring $\begin{pmatrix} R & M \\ 0 & R \end{pmatrix}$, and likewise additionally $T(R, R) \cong R[x]/\left\langle x^2 \right\rangle$.

We, moreover, note that the set of units of the trivial extension $T(R, M)$ is precisely
\[ U(T(R, M)) = T(U(R), M). \]

\medskip 

The next statement is worthwhile.

\begin{proposition}\label{prop2.22}
If the trivial extension $T(R,M)$ is a GUSC ring, then $R$ is a GUSC ring. The converse implication holds, provided $em=me$ for all $m\in M$ and $e\in Id(R)$.
\end{proposition}

\begin{proof}
First, we claim that if $(e,m)^{2}=(e,m)\in T(R,M)$, then $e^{2}=e$ and $m=0$. In fact, the equality $(e,m)^{2}=(e,m)$ automatically gives by comparison that $e^{2}=e$ and $em+me=m$. So, by hypothesis, we have $2em=m$, whence $2em=em$ and so $em=0$. Thus, $2em=0$ and hence $m=0$, proving the claim. Choose now $a\in R$, where $a\not\in U(R)$, so that $(a,0)\notin U(T(R,M))$ and thus $(a,0)$ is a USC element in $T(R,M)$. Then, $a$ is a USC element in $R$, because for otherwise, if $a$ has two different strongly clean decompositions, then $(a,0)$ has two different strongly clean decompositions contradicting our assumption.

For the reciprocal implication, choose arbitrary pair $(a,m)\in T(R,M)$, where $(a,m)\notin U(T(R,M))$, so that $a\not\in U(R)$ as for otherwise, if $a\in U(R)$, then $(a,m)\in U(T(R,M))$ with inverse $(a^{-1},-a^{-1}ma^{-1})$, a contradiction, and therefore by hypothesis $a$ is a USC element. Consequently, the element $(a,m)$ is USC, because for otherwise, if $(a,m)=(e,0)+(u,m)=(f,0)+(v,m)$, where $(e,0), (f,0)\in Id(T(R,M))$, $(u,m),(v,m)\in U(T(R,M))$ and
$(e,0)(u,m)=(u,m)(e,0)$, $(f,0)(v,m)=(v,m)(f,0)$, then we will have $$a=e+u=f+v,$$ where $e,f\in Id(R)$, $u,v\in U(R)$ and $eu=ue$, $fv=vf$. Finally, $a$ possesses two strongly clean decompositions which is impossible, establishing the statement.
\end{proof}

Two direct consequences sound as follows.

\begin{corollary}
If the trivial extension $T(R,R)$ is a GUSC ring, then $R$ is a GUSC ring. The opposite assertion holds when the ring $R$ is abelian.
\end{corollary}

\begin{corollary}
Let $R$ be an abelian ring, and let $M$ be a bi-module over $R$. Then, the following four statements are equivalent:
\begin{enumerate}
\item
$R$ is a GUSC ring.
\item
$T(R, R)$ is a GUSC ring.
\item
$R[x]/\left\langle x^2 \right\rangle$ is a GUSC ring.
\item
$R[[x]]/\langle x^2\rangle$ is a GUSC ring.
\end{enumerate}
\end{corollary}

Consider now $R$ to be a ring and $M$ to be a bi-module over $R$. Set $$BT(R,M) := \{ (a, m, b, n) | a, b \in R, m, n \in M \}$$ with addition defined component-wise and multiplication defined by $$(a_1, m_1, b_1, n_1)(a_2, m_2, b_2, n_2) = (a_1a_2, a_1m_2 + m_1a_2, a_1b_2 + b_1a_2, a_1n_2 + m_1b_2 + b_1m_2 +n_1a_2).$$ Then, $BT(R,M)$ becomes a ring which is isomorphic to $T(T(R, M), T(R, M))$. We also have the equality $$BT(R, M) =
\left\{\begin{pmatrix}
	a &m &b &n\\
	0 &a &0 &b\\
	0 &0 &a &m\\
	0 &0 &0 &a
\end{pmatrix} |  a,b \in R, m,n \in M\right\}.$$ In particular, we obtain the following isomorphism as rings $$\dfrac{R[x, y]}{\langle x^2, y^2\rangle} \rightarrow BT(R, R)$$ via the map $$a + bx + cy + dxy \mapsto
\begin{pmatrix}
	a &b &c &d\\
	0 &a &0 &c\\
	0 &0 &a &b\\
	0 &0 &0 &a
\end{pmatrix}.$$

\medskip

We, thereby, arrive at the following three corollaries.

\begin{corollary}
If $BT(R,M)$ is a GUSC ring, then $R$ is a GUSC ring. The converse is true, provided $em=me$ for all $m\in M$ and $e\in Id(R)$.
\end{corollary}

\begin{corollary}
If $BT(R,R)$ is a GUSC ring, then $R$ is a GUSC ring. The converse is valid when the ring $R$ is abelian.
\end{corollary}

\begin{corollary}
Let $R$ be an abelian ring, and let $M$ be a bi-module over $R$. Then, the following three statements are equivalent:
\begin{enumerate}
\item
$R$ is a GUSC ring.
\item
$BT(R, R)$ is a GUSC ring.
\item
$R[x, y]/\langle x^2, y^2\rangle$ is a GUSC ring.
\end{enumerate}
\end{corollary}

Let $R$ be a ring, and let $M$ be a bi-module over $R$ that is a non-unital ring in which the equalities $(mn)r=m(nr)$, $(mr)n=m(rn)$ and $(rm)n=r(mn)$ are fulfilled for all $m,n\in M$ and $r\in R$. Then, the ideal-extension $I(R,M)$ of $R$ by $M$ is defined as the additive abelian group $I(R,M)=R\oplus M$ with multiplication given by $$(r,m)(s,n)=(rs, rn+ms+mn).$$

\medskip

We are now in a position to establish the validity of the following statement.

\begin{proposition}\label{prop2.23}
If the ideal extension $I(R,M)$ is a GUSC ring, then $R$ is a GUSC ring. The converse holds provided the following conditions are satisfied:
\begin{enumerate}
\item[(a)]
If $e\in Id(R)$, then $em=me$ for all $m\in M$.
\item[(b)]
If $m\in M$, then $m+n+mn=0$ for some $n\in M$.
\end{enumerate}
\end{proposition}

\begin{proof}
The necessity is rather identical to the arguments given in Proposition \ref{prop2.22}.

To prove the sufficiency, we first know with the help of (a), (b) and \cite[Proposition7]{3} that each idempotent of $I(R,M)$ is of the type $(e,0)$, where $e^{2}=e$, as well as that each unit of $I(R,M)$ is of the type $(u,m)$, where $u\in U(R)$, $m\in M$.

Now, choosing $(a,m)\in I(R,M)$, where $(a,m)\not\in U(I(R,M))$, we observe that $a\not\in U(R)$. Thus, by hypothesis, the element $a$ is USC, and so $(a,m)$ is a USC element in $I(R,M)$, because for otherwise, if $(a,m)$ has two different strongly clean decompositions, then arguing as in the proof of Proposition \ref{prop2.22}, the element $a$ has two different strongly clean decompositions which is false. The assertion sustained.
\end{proof}

\begin{proposition}\label{2}
Let $R$ be a GUSC ring. Then, for any $n>2$, there does not exist $0\neq e\in Id(R)$ such that $eRe\cong M_{n}(S)$ for some ring $S$.
\end{proposition}

\begin{proof}
Let us assume the contrary, namely that there exists $0\neq e\in Id(R)$ such that $eRe\cong M_{n}(S)$ for some ring $S$. But, since by assumption $R$ is GUSC, it follows from Proposition \ref{1} that the corner subring $eRe$ is GUSC too, and hence $M_{n}(S)$ is GUSC as well. This, however, contradicts Theorem \ref{3}, sustaining our statement.
\end{proof}

Recall that a set $\{e_{ij} : 1 \le i, j \le n\}$ of non-zero elements of $R$ is said to be a {\it system of $n^2$ matrix units}, provided that $e_{ij}e_{st} = \delta_{js}e_{it}$, where $\delta_{jj} = 1$ and $\delta_{js} = 0$ for $j \neq s$. In this case, $e := \sum_{i=1}^{n} e_{ii}$ is an idempotent of $R$ and $eRe \cong M_n(S)$, where $$S = \{r \in eRe : re_{ij} = e_{ij}r, ~ \text{for all} ~ i, j = 1, 2, . . . , n\}.$$
Remember also that a ring $R$ is said to be {\it Dedekind-finite} if $ab=1$ insures $ba=1$ for any $a,b\in R$. In other words, all one-sided units in the ring are two-sided.

\medskip

What we can said now is the following observation.

\begin{proposition}\label{pro3.32}
Every GUSC ring is Dedekind-finite.
\end{proposition}

\begin{proof}
Suppose $R$ is a GUSC ring. If we assume the contrary, namely that $R$ is not a Dedekind-finite ring, then there exist two elements $a, b \in R$ such that $ab = 1$ but $ba \neq 1$. Setting for any $i,j>0$ that $$e_{ij}:= b^{i-1}(1-ba)a^{j-1} ~ {\rm and} ~ e:=\sum_{i=1}^{n}e_{ii}$$ yields $e_{11}=1-ba\not= 0$, and thus produces the required matrix units. Furthermore, a routine verification guarantees that there exists a non-zero ring $S$ such that $eRe \cong M_n(S)$. However, Proposition \ref{1} informs us that the corner $eRe$ is a GUSC ring, so that $M_n(S)$ has also be a GUSC ring, contradicting Theorem \ref{3}. This substantiates our claim.
\end{proof}	

Let $Nil_{*}(R)$ denote the {\it prime radical} (or, in other words, the {\it lower nil-radical}) of a ring $R$, i.e., the intersection of all prime ideals of $R$. We know that $Nil_{*}(R)$ is a nil-ideal of $R$.

Besides, it is also long known that a ring $R$ is {\it $2$-primal} if $Nil_{*}(R)$ consists precisely of all the nilpotent elements of $R$, that is, $Nil_{*}(R)=Nil(R)$. For instance, it is well known that reduced rings and commutative rings are both $2$-primal.

\medskip

Our next main result states thus.

\begin{theorem}\label{thm3.12}
Let $R$ be a $2$-primal, local and USC ring. Then, $M_2(R)$ is a GUSC ring.
\end{theorem}

\begin{proof}
Since $R$ is simultaneously local and USC, we can write that $R/J(R) \cong \mathbb{Z}_2$, so Example \ref{exmnew} suggests us that $M_2(R/J(R))$ is a GUSC ring.

On the other hand, since $R$ is both $2$-primal and USC, we may write $J(R) = Nil(R) = Nil_*(R)$, extracting that
$$M_2(R/J(R)) = M_2(R)/M_2(J(R)) = M_2(R)/M_2(Nil_*(R)) = M_2(R)/Nil_*(M_2(R)).$$
Moreover, as $Nil_*(M_2(R))$ is a nil-ideal, we discover that $M_2(R)$ is a GUSC ring, as inferred.
\end{proof}

In the above theorem, the condition of being local is {\it not} redundant. In order to demonstrate that, suppose that $R = \mathbb{Z}_2 \times \mathbb{Z}_2$; then, $R$ is both commutative (whence $2$-primal) and USC, but definitely $M_2(R)$ is {\it not} a GUSC ring as an obvious check shows.

In the same vein, it is {\it not} possible to interchange the condition of being USC with GUSC in the last theorem. Indeed, to illustrate that such a replacement really cannot be happen, assuming $R = \mathbb{Z}_3$, then $R$ is simultaneously commutative (whence $2$-primal), GUSC and local, but an easy verification gives that $M_2(R)$ is {\it not} a GUSC ring, as suspected.

\medskip

The next assertion somewhat sounds rather strange, but is absolutely true.

\begin{lemma}\label{lem3.13}
Let $M_2(R)$ be a GUSC ring. Then, $R$ is a USC ring.
\end{lemma}

\begin{proof}
Given $a \in R$, then one calculates about the matrix $$A=\begin{pmatrix}
		a & 0 \\ 0 & 0
\end{pmatrix} \not\in U(M_2(R)).$$ Thus, $A$ is a USC element in $M_2(R)$. Therefore, one concludes that $a$ is a USC element in $R$, as expected.
\end{proof}

Our further principal result sounds like this.

\begin{theorem}\label{lem3.15}
Let $R$ be a ring. Then, the following two points are equivalent for a semi-simple ring $R$:
\\
(i) $R$ is a GUSC ring.
\\
(ii) $R$ is either a local ring, or a USC ring, or $R=M_2(\mathbb{Z}_2)$.
\end{theorem}

\begin{proof}
(i) $\Rightarrow$ (ii). Since $R$ is a semi-simple ring, we may write $R = \prod_{i=1}^{m} M_{n_i}(D_i)$, where each component $M_{n_i}(D_i)$ is a matrix ring over a division ring $D_i$.

If, for a moment, $m=1$, then combining Example \ref{exmnew} and Theorem \ref{thm3.12}, we detect that either $R = D_1$ or $R = M_2(\mathbb{Z}_2)$.
	
If, however, $m>1$, then employing Proposition \ref{lem2.3}, each of the rings $M_n(D)$ must be USC, which in turn means that, for any $i$, the equality $M_{n_i}(D_i) = \mathbb{Z}_2$ holds.
\\
(ii) $\Rightarrow$ (i). If $R$ is a USC ring, then it is apparent that $R$ is GUSC. If now $R$ is local, then $R$ has to be GUSC, as formulated.
\end{proof}

\begin{corollary}\label{pronew}
Let $R$ be a ring. Then, the following two issues are equivalent for an abelian ring $R$:
\\
(i) $R$ is a GUSC ring.
\\
(ii) $R$ is either a local, or a USC ring.
\end{corollary}

Furthermore, we recall that a group $G$ is a {\it $p$-group} if each element of $G$ has order which is a power of $p$, where $p$ is an arbitrary prime. A group $G$ is called {\it locally finite} if every finitely generated subgroup is finite.

As usual, for an arbitrary ring $R$ and an arbitrary group $G$, the symbol $RG$ designates the {\it group ring} of $G$ over $R$.

\medskip

Two more statements occurred.

\begin{proposition}\label{exam3.5}
Let $C_{n}$ denote the cyclic group of order $n$. If $n\geq 3$ is odd and $R$ is a boolean ring, then $RC_{n}$ is {\it not} a GUSC ring.
\end{proposition}

\begin{proof}
According to \cite [Example 4.7]{1}, we have $1+g^{n-1}\notin U(RG)$. Also, we see that $g^{n-1}, x+g^{n-1}\in U(RG)$, where $$x := g + g^2 + \cdots + g^{n-1}$$ and $$1+g + g^2 + \cdots + g^{n-1}$$ is an idempotent. So, one writes that $$1+g^{n-1}=(1+g + g^2 + \cdots + g^{n-1})+(x+g^{n-1})$$ are two strongly clean decompositions for $1+g^{n-1}$ leading to a contradiction, as required.
\end{proof}

In closing, we formulate the following.

\begin{proposition}
Let $R$ be a ring in which $2$ is invertible, and let $G=\{1, g\}$ be a group. Then, $RG$ is GUSC if, and only if, $R$ is USC.
\end{proposition}

\begin{proof}
By hypothesis, we easily have $RG\cong R\times R$ via the map $a+bg \leftrightarrow (a+b, a-b)$. Letting $R$ be USC, we then see that $R\times R$ is too USC in conjunction with \cite[Example 2]{2}. Hence, $RG$ is USC and so it is GUSC.

Conversely, let $RG$ be GUSC. Then, the direct product $R\times R$ is GUSC, and hence a consultation with Proposition~\ref{prop2.5} is a guarantor that $R$ is USC, as concluded.
\end{proof}

\section{Open Questions}

We close our current work with the following two challenging problems of some interest and importance.

\begin{problem}
If $R$ is a semi-potent ring, does it follows that $R$ is GUSC if, and only if, $R$ is GUC?
\end{problem}

\begin{problem}
Suppose that $R$ is an arbitrary ring. Find a suitable criterion in terms of the former ring $R$ and its divisions when the power series ring $R[[x]]$ is GUSC.
\end{problem}

\medskip

\noindent{\bf Acknowledgements:} The authors are very thankful to both specialists who refereed this paper for their expert reports containing numerous valuable comments and suggestions that improved on the initial presentation overall.

\medskip

\noindent{\bf Declarations:} The authors declare {\bf no} any conflict of interests as well as {\bf neither} data {\bf nor} AI tools were used while preparing and writing the current manuscript.

\end{document}